\documentclass[12pt,reqno]{amsart}
\usepackage{graphics}
\usepackage{enumerate}
\usepackage{tikz}
\usetikzlibrary{positioning,decorations.pathreplacing}
\usepackage{graphicx} 
\usepackage{psfrag} 
\usepackage{amssymb,a4wide,amsthm,amsmath,amsfonts}
\usepackage{mathrsfs}
\usepackage{color}
\usepackage{appendix}

\theoremstyle{plain}
\newtheorem{theorem}{Theorem}[section]

\newtheorem{thmx}{Theorem}

\theoremstyle{remark}
\newtheorem{remark}[theorem]{Remark}

\theoremstyle{plain}
\newtheorem{corollary}[theorem]{Corollary}
\newtheorem{lemma}[theorem]{Lemma}
\newtheorem{proposition}[theorem]{Proposition}
\newtheorem{definition}[theorem]{Definition}

\numberwithin{equation}{section}
\newtheorem{fact}[theorem]{Fact}

\newtheorem{question}{Question}

\newcommand{\Rea}{\mathbb{R}}
\newcommand{\R}{\mathbb{R}}
\newcommand{\Nat}{\mathbb{N}}
\newcommand{\N}{\mathbb{N}}
\newcommand{\Int}{\mathbb{Z}}

\newcommand{\F}{\mathcal{F}}
\newcommand{\Sph}{\mathbb{S}}

\newcommand{\Hy}{\mathbb{H}}

\newcounter{probcount}

\begin{document}
\title{Approximation properties in Lipschitz-free spaces over groups}
\author{Michal Doucha}
\address{Institute of Mathematics\\
Czech Academy of Sciences\\
\v Zitn\'a 25\\
115 67 Praha 1\\
Czech Republic}
\email{doucha@math.cas.cz}
\author{Pedro L. Kaufmann}
\address{Instituto de Ci\^encia e Tecnologia da Universidade federal de S\~ao Paulo, Av. Cesare Giulio Lattes, 1201, ZIP 12247-014 S\~ao Jos\'e dos Campos/SP, Brasil}
\email{plkaufmann@unifesp.br}
\thanks{P. L. Kaufmann was supported by grants 2017/18623-5 and  2016/25574-8, São Paulo Research Foundation (FAPESP). M. Doucha was supported by the GA\v{C}R project 19-05271Y and RVO: 67985840.} 
\keywords{Lipschitz-free space, Arens-Eells space,  metric approximation property, Schauder basis, compact group, homogeneous space}
\subjclass[2010]{46B20 (primary); 22C05, 42A24 (secondary)}
\begin{abstract}
We study Lipschitz-free spaces over compact and uniformly discrete metric spaces enjoying certain high regularity properties - having group structure with left-invariant metric. Using methods of harmonic analysis we show that, given a compact metrizable group $G$ equipped with an arbitrary compatible left-invariant metric $d$, the Lipschitz-free space over $G$, $\F(G,d)$, satisfies the metric approximation property. We show also that, given a finitely generated group $G$, with its word metric $d$, from a class of groups admitting a certain special type of combing, which includes all hyperbolic groups and Artin groups of large type, $\F(G,d)$ has a Schauder basis. Examples and applications are discussed. In particular, for any net $N$ in a real hyperbolic $n$-space $\mathbb{H}^n$, $\F(N)$ has a Schauder basis. 
\end{abstract}
\maketitle

\section{Introduction}
Lipschitz-free spaces form by now a fundamental class of Banach spaces, whose study has been revitalized since the appearance of the seminal paper of Godefroy and Kalton (\cite{godefroy2003lipschitz}). There are two main important properties that both characterize these spaces. Namely, they are free objects in the category of Banach spaces over the metric spaces. Second, they are canonical isometric preduals to the Banach spaces of pointed Lipschitz real-valued functions. Another appealing feature is that their study connects Banach space theory to several other areas of mathematics, including optimal transport and geometry, and, as we demonstrate here, also harmonic analysis. We recall some basic facts about Lipschitz-free spaces in Section~\ref{section:preliminaries}.

Approximation properties in Lipschitz-free spaces have been one of the main research directions since the publication of \cite{godefroy2003lipschitz}.  It has become clear since then that there are metric spaces such that the corresponding Lipschitz-free space does not have the approximation property, since by \cite[Theorem 5.3]{godefroy2003lipschitz}, a Banach space $X$ has the bounded approximation property if and only if the Lipschitz-free space $\F(X)$ does. The attention was therefore shifted to certain amenable classes of metric spaces, in particular compact metric spaces and, to some extent, also to uniformly discrete metric spaces. The compact metric case is particularly important since it has been shown by Godefroy in \cite{God2015} that the bounded approximation property of Lipschitz-free space over a compact metric space $M$ is equivalent to the existence of linear almost extension operators of Lipschitz functions over subsets of $M$, a subject currently receiving high attention in both geometry and computer science (see e.g. \cite{BruBru} and \cite{LN05}). The question whether Lipschitz-free space over any uniformly discrete metric space has the bounded approximation property is perhaps the most serious and still open, we refer to \cite[Question 1]{godefroy2014free} for a motivation and to \cite{Kalton} for the proof that such a space has the approximation property. Regarding compact metric spaces, the first compact metric space such that the corresponding Lipschitz-free space fails the approximation property has been found in \cite{godefroy2014free}, and later, even a compatible metric on the Cantor space has been found so that the Lipschitz-free space lacks the approximation property (see \cite{HaLaPe}). Positive results have been however obtained in \cite{Dalet-compact} and \cite{dalet2015free} when one restricts to countable compact, resp. proper metric spaces.

The goal of this paper is to consider certain fundamental classes of compact metric spaces, resp. uniformly discrete metric spaces, which are amenable to methods of harmonic analysis, resp. geometry. Namely, compact groups with left-invariant (or right-invariant) metrics, resp. finitely generated groups with word metrics. It turns out that harmonic analytic methods, resp. certain combinatorial and geometric methods, go hand in hand with our goal of showing approximation properties in Lipschitz-free spaces over compact metric groups, resp. finitely generated groups.

In the compact group case we obtain a satisfactory definitive solution.
\begin{thmx}\label{thm:intro1}
Let $G$ be a compact metrizable group with an arbitrary compatible left-invariant (or right-invariant) metric $d$. Then $\F(G,d)$ has the metric approximation property.
\end{thmx}
Just to show a meager application, we recall that there has been interest in for which compatible metrics of the Cantor space the corresponding Lipschitz-free space has some approximation property. Godefroy and Ozawa show in \cite{godefroy2014free} that for certain `small Cantor spaces', the free space has the metric approximation property. On the other hand, H\' ajek, Lancien, and Perneck\' a show in \cite{HaLaPe} that there are `fat Cantor spaces' for which the free space does not have the approximation property. We recall that there is a very large and thoroughly studied class of compact (metrizable) groups, the \emph{profinite (metrizable) groups}, which are inverse limits of finite groups. So in the infinite metrizable case, they are topologically totally disconnected uncountable metrizable spaces without isolated points - therefore homeomorphic to the Cantor space (see the monograph \cite{profinite} for more information on profinite groups). It turns out that for any compatible and left-invariant metric on any such group structure on the Cantor space we get a free space with the metric approximation property.

In case of finitely generated groups, the metric approximation property follows from known results (see Section~\ref{section:fingengrps}), so we aim for much stronger property, having the Schauder basis, at the cost of having less general result that applies just to a certain subclass of finitely generated groups. We state the result and postpone the definition of the new notions to the corresponding section.
\begin{thmx}\label{thm:intro2}
Let $G$ be a shortlex combable group with its word metric $d$. Then $\F(G,d)$ has the Schauder basis {\it (see Theorem~\ref{thm:shortlex})}.
\end{thmx}
We mention that the theorem applies in particular to hyperbolic groups and Artin groups of large type. One of the applications (see Corollary~\ref{cor:hyperbolicnet}) is that the Lipschitz-free space over any net in the real hyperbolic $n$-space $\Hy^n$ has the Schauder basis. \\

The paper is organized as follows. In Section \ref{section:preliminaries} we present a characterization of the $\lambda$-bounded approximation property tailored for Lipschitz-free spaces (Proposition \ref{tool}). In Section \ref{section:cpt} we prove Theorem \ref{thm:intro1}; first we tackle Lie groups using harmonic analysis tools in Subsection \ref{subsection:cptLie}, then in Subsection \ref{subsection:generalCpt} we prove the general case by approximating compact groups by compact Lie ones. In this last subsection we also generalize Theorem \ref{thm:intro1} to compact homogeneous spaces equipped with quotient metrics (Theorem \ref{thm:homogeneousspace}). Section \ref{section:fingengrps} is dedicated to finitely generated groups; we prove Theorem \ref{thm:intro2} and provide some examples and applications. We conclude with some remarks and questions in Section \ref{section:problems}, and presenting in Appendix \ref{appendSphere} a generalization of Theorem \ref{thm:homogeneousspace} for the specific case of the euclidean sphere.

\section{Preliminaries}\label{section:preliminaries}
\subsection{Lipschitz-free spaces}
Let $M$ be a metric space and $0$ be some distinguished point in $M$. Let $\mathrm{Lip}_0(M)$ denote the Banach space of real-valued Lipschitz functions defined on $M$ which vanish at $0$, equipped with the  norm $\|\cdot\|_{\mathrm{Lip}}$ which assigns to each function its minimal Lipschitz constant. The Lipschitz-free space over $M$, denoted by $\F(M)$, is the canonical isometric predual of $\mathrm{Lip}_0(M)$ given by the closed linear span of $\{\delta(x):x\in M\}$ in $\mathrm{Lip}_0(M)^*$, where each $\delta(x)$ is the evaluation functional defined by $\delta(x)(f):=f(x)$. This gives $\mathrm{Lip}_0(M)$ a $w^*$ topology which coincides, on bounded sets of $\mathrm{Lip}_0(M)$, with the topology of pointwise convergence. $\F(M)$ satisfies a powerful universal property: given a Banach space $X$ and a Lipschitz function $F:M\to X$ vanishing at $0$, there exists a unique bounded linear operator $T:\F(M)\to X$ such that $T\circ \delta = F$. Its operator norm coincides with the Lipschitz constant of $F$. We refer to Weaver's book \cite{weaver1999lipschitz} for a thorough introduction to the subject. There, Lipschitz-free spaces are denominated Arens-Eells spaces. 
\subsection{Verifying if a Lipschitz-free space has the BAP}
In this subsection we present a characterization of $\lambda$-bounded approximation property suited for Lipschitz-free spaces (Proposition \ref{tool} below). Let us briefly recall and comment the definition of this property: 

\begin{definition}[Bounded Approximation Property]

Let $X$ be a Banach space, and $\lambda\geq 1$. We say that $X$ has the $\lambda$-approximation property ($\lambda$-BAP for short) if one of the following equivalent assertions holds: 

\begin{enumerate}
\item For each compact subset $K$ of $X$ and each $\epsilon>0$, there is a $\lambda$-bounded finite rank operator $T$ on $X$ such that $\|Tx-x\|<\epsilon$, for each $x\in K$. 
\item For each finite subset $F$ of $X$ and each $\epsilon>0$, there is a $\lambda$-bounded finite rank operator $T$ on $X$ such that $\|Tx-x\|<\epsilon$, for each $x\in F$.
\item There is a $\lambda$-bounded net of finite rank operators $(T_\alpha)$ on $X$ such that $\langle \varphi, T_\alpha x\rangle \rightarrow \langle \varphi,x\rangle$ for each $\varphi\in X^*$ and each $x\in X$ (that is, $T_\alpha$ converges to the identity operator in the \emph{weak operator topology}).
\end{enumerate}

If a Banach space $X$ has the 1-BAP, we say that $X$ has the metric approximation property (MAP for short). 

\label{defBAP}
\end{definition}

Formulations (1) and (2) are classic; their equivalence with (3) is shown for instance in \cite{kim2008characterizations}. Recall that a Banach space $X$ has the $\lambda$-BAP if and only if, for each $\delta>0$, $X$ has the $(\lambda+\delta)-BAP$. To see this, fix a compact set $K\subset X$ and $\varepsilon>0$, and take $\delta>0$ small enough so that $M\delta(\lambda+\delta)/\lambda < \epsilon/2$, where $M=\sup_{x\in K}\|x\|$. Let $T$ be a finite rank, $(\lambda+\delta)$-bounded operator on $X$ such that $\|Tx - x\|<\epsilon/2$, for all $x\in K$.  Then it is immediately verified that the $\lambda$-bounded operator $S=\lambda T/\|T\|$ satisfies, for each $x\in K$, $\|Sx - x\| <\varepsilon.$ \\

In what follows $\mathrm{Lip}(M)$ denotes the space of real-valued functions defined on the metric space $M$. We will still use $\|f\|_{\mathrm{Lip}}$ to denote the Lipschitz constant of $f\in \mathrm{Lip}(M)$, keeping in mind that $\|\cdot\|_{\mathrm{Lip}}$ defines only a seminorm in $\mathrm{Lip}(M)$. 

\begin{proposition}[Characterization of BAP for Lipschitz-free spaces]

Let $K$ be a compact metric space and  $\lambda\geq 1$. The following assertions are equivalent: 
\begin{enumerate}
\item $\F(K)$ has the $\lambda$-BAP; 
\item for each $\varepsilon>0$ there is a net $T_\alpha$ of bounded operators on $C(K)$ such that 
\begin{enumerate}
\item $T_\alpha$ are of finite rank,
\item each $T_\alpha$ maps Lipschitz functions to Lipschitz functions, 
\item $\|T_\alpha f\|_{\mathrm{Lip}}\leq (\lambda+\varepsilon)\|f\|_{\mathrm{Lip}}$, for each $\alpha$ and each $f\in \mathrm{Lip}(K)$, and
\item $T_\alpha f$ converges pointwise to $f$, for each $f\in \mathrm{Lip}(K)$.
\end{enumerate}

\end{enumerate}
\label{tool}
\end{proposition}

For the proof we will need the following: 

\begin{lemma}

Let $K$ be a compact metric space, fix  $0\in M$, and let $T$ be a bounded operator on $C(K)$ such that $T(B_{\mathrm{Lip}_0(K)})\subset A.B_{ \mathrm{Lip}_0(K)}$ for some $A>0$.  Then $T$ restricted to $\mathrm{Lip}_0(K)$, when seen as an operator on $\mathrm{Lip}_0(K)$, is $w^*$-continuous. 

\label{Trest}
\end{lemma}

\begin{proof} Consider $S=T|_{\mathrm{Lip}_0(K)}$ as an operator on $\mathrm{Lip}_0(K)$. Let $f_\nu$ be a bounded net $w^*$-converging to $f$ in $B_{\mathrm{Lip}_0(K)} $. Since on bounded sets the $w^*$ topology coincides with the topology of pointwise convergence and $f_\nu$ are equicontinuous, it follows that $f_\nu$ converges to $f$ uniformly. By the norm continuity of $T$, $Sf_\nu$ converges to  $Sf$ uniformly and, since $\{Sf_\nu\}$ is $A$-bounded in the Lipschitz norm, $Sf_\nu$ $w^*$-converges to $Sf$. That is, $S$ is $w^*$-continuous when restricted to $B_{\mathrm{Lip}_0(K)}$. It follows by Banach-Dieudonn\'e's theorem that $S$ is $w^*$-continuous in $\mathrm{Lip}_0(K)$. \end{proof}

\begin{proof}[Proof of Proposition \ref{tool}] ((1)$\Rightarrow$(2)) Let $T_\alpha$ be a net of $\lambda$-bounded finite rank operators on $\F(K)$ converging to the identity in the weak operator topology. Their adjoints $T_\alpha^*:\mathrm{Lip}_0(K)\rightarrow \mathrm{Lip}_0(K)$ can be extended to operators on $C(K)$ that clearly satisfy properties (a)--(d).\medskip

((2)$\Rightarrow$(1)) Choose some $0\in K$, and $T_\alpha$ satisfying (a)--(d). It is easily checked that the operators $S_\alpha$ defined on $C(K)$ by
$$
S_\alpha (f)(x) = T_\alpha (f)(x) - T_\alpha (f)(0)
$$
are bounded, satisfy the same assumptions (a)--(d) and $S_\alpha$ are moreover ($\lambda+\varepsilon$)-bounded operators on $\mathrm{Lip}_0(K)$. Each $S_\alpha$ thus satisfies the assumptions of Lemma \ref{Trest}, and it follows that there are finite-rank, ($\lambda+\varepsilon$)-bounded operators $R_\alpha$ on $\F(K)$ with $R_\alpha^*=S_\alpha$, that is: 
$$
\langle S_\alpha f,\gamma\rangle = \langle f, R_\alpha \gamma\rangle,\,\forall f\in \mathrm{Lip}_0(K),\,\forall \gamma \in \F(K).
$$
Now since $S_\alpha f$ is a bounded net in $\mathrm{Lip}_0(K)$ converging pointwise to $f$, it must converge $w^*$, thus $R_\alpha$ converges to $Id_{\F(M)}$ in the weak operator topology. The conclusion follows from the remark after Definition \ref{defBAP}. \end{proof}

We point out that the above characterization can be generalized to an arbitrary metric space $M$ if we forget the part of $T_\alpha$ being bounded operators on some $C(K)$, and assume that $T_\alpha$ are ($\lambda+\varepsilon)$-bounded operators on $\mathrm{Lip}(M)$ which are moreover pointwise continuous. We omit the proof, which follows the same lines.

\section{Approximation properties of Lipschitz-free spaces over compact groups}\label{section:cpt}

The classical Birkhoff-Kakutani theorem says that a topological group is metrizable if and only if it is metrizable by a left-invariant (equivalently right-invariant) metric if and only if it is Hausdorff and first countable. Compatible left-invariant metrics on a fixed metrizable group are easily seen to be unique up to uniform equivalence. However, a group can admit compatible and left-invariant metrics which are not bi-Lipschitz equivalent. It is even possible that the corresponding Lipschitz-free spaces are not isomorphic. For example, if $d$ is the arc length metric on the torus $\mathbb{T}$, the Lipschitz-free space $\F(\mathbb{T},d)$ is isomorphic to $L^1$, but if we substitute $d$ by $d^\alpha$ with $0<\alpha<1$ (in the process often referred to as \emph{snowflaking}),  $\F(\mathbb{T},d^\alpha)$ is isomorphic to $\ell_1$ (see \cite[Theorem 6.5]{Kalton}).

In this section we focus on compact metrizable groups (or equivalently, compact first countable groups) and the issue of many non-Lipschitz equivalent left-invariant metrics does not bother us. We use methods of harmonic analysis that are robust enough to make our proofs work for arbitrary compactible left or right-invariant metrics on any compact metrizable group. The main goal of the section is to prove Theorem~\ref{thm:intro1}, however we will have something to say also about free spaces over certain homogeneous spaces of compact groups (see Theorem~\ref{thm:homogeneousspace}).

The section is divided into two subsections. One is dealing with compact Lie groups where the machinery of harmonic analysis is used. The other is dealing with general compact metrizable groups using the fact each such a group can be approximated by compact Lie groups. This approximation then lifts to the corresponding Lipschitz-free space; indeed, we shall show that Lipschitz-free space over a compact metric group can be approximated by Lipschitz-free spaces over compact Lie groups.

\subsection{Lipschitz-free spaces over compact Lie groups}\label{subsection:cptLie}\hfill\bigskip
\label{cpt}

From this point on, unless stated otherwise, we always assume that each locally compact group $G$ is equipped with a left-invariant Haar measure, denoted by $\mu$. When $G$ is compact, $\mu$ is assumed to be bi-invariant and of course normalized.

Let us start some basic definitions and facts from representation theory of compact groups, which can be found in any standard textbook on harmonic analysis or representation theory of compact groups.

Let $G$ be a compact group. A \emph{unitary representation} $U$ is a strongly continuous group homomorphism $x\in G\mapsto U_x \in B(H_U)$ (i.e. on $B(H_U)$ we consider the strong operator topology), where $H_U$ is a complex Hilbert space, such that each operator $U_x$ is unitary on $H_U$. Such unitary representation is said to be \emph{irreducible} if the only invariant subspaces of $H_U$, i.e. subspaces preserved by all $U_x$, are $\{0\}$ and $H_U$. In our specific case where $G$ is compact, all continuous irreducible unitary representations $x\in G\mapsto U_x \in B(H_U)$ satisfy $\dim H_U <\infty$ (see \cite[Theorem 22.13]{hewitt2012abstract}). Given an irreducible unitary representation $U:G\rightarrow B(H_U)$ and non-zero vectors $\xi,\eta\in H_U$, the function of the form $$x\in G\mapsto \langle U_x \xi,\eta\rangle \in \mathbb{C}$$ is called a \emph{matrix coefficient} (associated to $U$). If $\xi=\eta$, we call such matrix coefficient a \emph{positive-definite function}. If $H_U$ is equipped with a fixed orthonormal basis $\{\zeta^U_1,...,\zeta^U_{d_U}\}\subseteq H_U$, then the matrix coeeficients $$\varphi^U_{jk}(x) = \langle U_x \zeta^U_j,\zeta^U_k\rangle$$ are called \emph{coordinate functionals} (associated to $U$ and $\{\zeta^U_1,...,\zeta^U_{d_U}\}$). They satisfy the following property, as a consequence of \cite[Theorem 27.20]{hewitt2013abstract}:
\begin{eqnarray}
\forall f\in C(G)\,\forall x\in G,\, f\ast \varphi^U_{jk}(x) = \sum_{r=1}^{d_U} \int f(y)\,\overline{\varphi^U_{rj}(y)}\,d\lambda(y)\, \varphi^U_{rk}(x).
\label{coordfunct}
\end{eqnarray}

A \emph{trigonometric polynomial} on $G$ is a linear combination of matrix coefficients. It is straightforward to see that all trigonometric polynomials are in fact linear combinations of coordinate functionals, independently of the choice of bases $\{\zeta^U_1,...,\zeta^U_{d_U}\}$, which are assumed to be fixed. Thus, for any given trigonometric polynomial $P$, there is a finite set $F$ of continuous irreducible unitary representations such that $P$ is in the linear span of  $\{\varphi^U_{jk}| U \in F,\, j,k = 1,...,d_U\}$. It follows from (\ref{coordfunct}) that the operator $f\in C(G) \mapsto f\ast P\in C(G)$ has its range contained in the linear span of $\{\varphi^U_{jk}| U \in F,\, j,k = 1,...,d_U\}$, thus in particular it is of finite rank.\bigskip

The following proposition is our crucial tool from harmonic analysis.
\begin{proposition}\label{crucialharmonic}
Suppose that $G$ is a compact Lie group. Then there exists a sequence $F_n$ of positive real functions on $G$ satisfying
\begin{enumerate}
\item each $F_n$ is a positive definite \emph{central} (commutes under convolution with any function in $L_1(G)$) trigonometric polynomial,
\item $F_n(g^{-1}) = F_n(g)$, $g\in G$, for each $n$, 
\item for each $n$, $\int F_n \,d\lambda =1$, and
\item $f\ast F_n(x) \rightarrow f(x)$ $\lambda$-almost everywhere for every $f\in L_p(G),\,1\leq p <\infty$.  
\end{enumerate}
\end{proposition}
\begin{proof}
By \cite[Corollary IV.4.22]{Knapp}, every compact Lie group is isomorphic to a matrix group, so we may suppose that $G\subseteq \mathrm{GL}(n,\mathbb{C})$ for some $n\in \N$. By the following standard `unitarization trick', we may assume that $G$ is a (necessarily closed) subgroup of the unitary group $U(n)$: let $\langle \cdot,\cdot\rangle'$ be an arbitrary inner product on $\mathbb{C}^n$. Define a new inner product $\langle \cdot,\cdot\rangle$ by setting $$\langle \xi,\eta\rangle:=\int_{g\in G} \langle g\xi,g\eta\rangle'd\mu(g),$$ where $\mu$ is an invariant probability Haar measure on $G$. It is standard and straightforward to check that $\langle \cdot,\cdot\rangle$ is still an inner product, which is moreover invariant by the action of $G$. That implies that $G$ is a subgroup of a finite-dimensional unitary group. Now by \cite[Theorem 44.29]{hewitt2013abstract}, $G$ satisfies all conditions needed to apply directly \cite[Theorem 44.25]{hewitt2013abstract}, which gives us positive real functions $F_n$ on $G$ satisfying the conditions (1)--(4).
\end{proof}

Another ingredient we will need is  the following version of Young's convolution inequality, suitable for Lipschitz functions defined on a locally compact group: 

\begin{lemma}[Young's inequality]
Let $G$ be a locally compact group equipped with a compatible left-invariant metric $d$. 
Suppose that $f\in L^1(G)$ and $g\in \mathrm{Lip}(G)$. Then $f\ast g\in \mathrm{Lip}(G)$ and 
$$
\|f\ast g \|_{\mathrm{Lip}} \leq \|f\|_{L^1}\|g\|_{\mathrm{Lip}}.
$$

\label{lipconv}

\begin{proof}Given arbitrary $x,y\in G$, 
\begin{align*}
|f*g(x) - f*g(y)| & = \left| \int f(z)[g(z^{-1}x) - g(z^{-1}y)]\,d\mu(z) \right| \\
& \leq \int |f(z)| |g(z^{-1}x) - g(z^{-1}y)|\,\mu(z)\\
& \leq \|g\|_{\mathrm{Lip}}\int |f(z)| d(z^{-1}x,z^{-1}y)\, \mu(z)\\
& \leq \|g\|_{\mathrm{Lip}}\int |f(z)| d(x,y) \,\mu(z)\\
&\leq \|f\|_{L^1}\|g\|_{\mathrm{Lip}} d(x,y).
\end{align*}
\end{proof}

\end{lemma}

We are now ready to prove the main result from this subsection. We emphasize that in the following we are equipping a Lie group with an \emph{arbitrary} left-invariant metric inducing its topology, not  Riemannian metric as it is common in Lie theory.

\begin{theorem}

Suppose that $G$ is a compact Lie group  equipped  with a compatible left-invariant metric. Then $\F(G)$ has the MAP. 

\label{proptorus}
\end{theorem}

\begin{proof} Define $T_n: C(G) \rightarrow C(G)$ by $T_n(f) = f\ast F_n$, where $F_n$ are the functions established in Theorem (\ref{crucialharmonic}). It suffices to see that these operators satisfy conditions (a)--(d) from Proposition (\ref{tool}) with $\lambda=1$, from which follows that $\F(G)$ satisfies the MAP. Indeed, each $T_n$ has finite rank, since $F_n$ is a trigonometric polynomial. Thus condition (a) is satisfied. Young's inequality gives us that $\|T_n f \|_{\mathrm{Lip}} \leq \|F_n\|_{L^1}\|f\|_{\mathrm{Lip}}\leq \|f\|_{\mathrm{Lip}}$, for each $f\in Lip(G)$. Thus condition (b) and (c) are satisfied with constant $\lambda = 1$. (d) follows from condition (4) and the fact that, for each fixed $f\in Lip(G)$, $\{T_n f\}_n$ is equicontinuous, which yields uniform convergence since $G$ is compact. \end{proof}

\subsection{Lipschitz-free spaces over general compact metric groups and homogeneous spaces}\label{subsection:generalCpt}\hfill\bigskip

We start with some remarks on quotient metrics on homogeneous spaces. Let $G$ be a topological group equipped with a compatible left-invariant metric $d$ and let $H$ be a closed subgroup. We want to define a quotient metric on the homogeneous space $G/H$ of left-cosets of $H$. There are two cases. Either $H$ is normal, so $G/H$ is a group. Then the formula 
\begin{eqnarray}\label{quotmetric}
D(gH,fH):=\inf_{h_1,h_2\in H} d(gh_1,fh_2)
\end{eqnarray}
defines a compatible left-invariant metric on the quotient group $G/H$, which we refer to as \emph{quotient metric}. We leave the easy verification to the reader. Or $H$ is not normal, so $G/H$ is just a $G$-homogeneous space, i.e. a homogeneous space equipped with a continuous transitive action of $G$. Then the same formula $D(gH,fH):=\inf_{h_1,h_2\in H} d(gh_1,fh_2)$ does not in general define a metric; it does define a compatible $G$-invariant metric provided that $d$ is additionally right $H$-invariant, i.e. $d(g,f)=d(gh,fh)$ for $g,f\in G$ and $h\in H$. We show this. We only verify the triangle inequality, the compatibility and $G$-invariance are easier and left to the reader. Pick $g_1,g_2,g_3\in G$ and let us check that $D(g_1H,g_3H)\leq D(g_1H,g_2H)+D(g_2H,g_3H)$. Choose an arbitrary $\varepsilon>0$ and and some $h_1,h_2, h'_2,h_3\in H$ such that $D(g_1H,g_2H)\geq d(g_1h_1,g_2h_2)-\varepsilon$ and $D(g_2H,g_3H)\geq d(g_2h'_2,g_3h_3)-\varepsilon$. Then $$D(g_1H,g_3H)\leq d(g_1h_1,g_3h_3(h'_2)^{-1}h_2)\leq d(g_1h_1,g_2h_2)+d(g_2h_2,g_3h_3(h'_2)^{-1}h_2)=$$ $$d(g_1h_1,g_2h_2)+d(g_2h'_2,g_3h_3)\leq D(g_1H,g_2H)+D(g_2H,g_3H)-2\varepsilon.$$ Since $\varepsilon$ was arbitrary, we are done.

We note that when $G$ is a metrizable group and $H$ is a compact subgroup, then a compatible left-invariant and right $H$-invariant metric on $G$ always exists. Indeed, let $d$ be an arbitrary compatible left-invariant metric on $G$. We define, for $g,f\in G$, $D(g,f):=\max_{h\in H} d(gh,fh)$. Alternatively, using a normalized invariant Haar measure $\mu$ on $H$, we can define $D$ by averaging as follows: $D(g,f):=\int_H d(gh,fh)d\mu(h)$. We leave to the reader to check that both formulas define a compatible left-invariant and right $H$-invariant metrics.\medskip

Our main tool in this subsection will be the following proposition. We note that simultaneously while writing this paper, the content of the proposition is being developed into a more general form in \cite{AACD}.

\begin{proposition}\label{prop:projection}
Let $G$ be a topological group equipped with a compatible metric $d$ and a compact subgroup $H$. Suppose, additionally, that at least one of the following conditions holds: 
\begin{enumerate}[(i)]
    \item $d$ is left-invariant and $H$ is normal,
    \item $d$ is right-invariant and $H$ is normal, or
    \item $d$ is left-invariant and right $H$-invariant. 
    \end{enumerate}
Then there exists a norm one projection $P:\F(G,d)\rightarrow \F(G,d)$ ranging onto a linearly isometric copy of $\F(G/H,D)$, where $D$ is the quotient metric as defined in \eqref{quotmetric}.
\end{proposition}
\begin{proof}
By the discussion preceding the statement of the proposition, it is verified that $D$ is a well-defined metric. Let $\mu$ be the normalized invariant Haar measure on $H$. If (ii) or (iii) holds, we define a map $P':G\rightarrow \F(G)$ by setting for any $g\in G$ 
\begin{eqnarray}P'(g):=\int_H \delta(g\cdot h)-\delta(h)d\mu(h).
\label{P'}
\end{eqnarray}In case only (i) holds, we define for any $g\in G$ $$P'(g):=\int_H \delta(h\cdot g)-\delta(h)d\mu(h).$$ We will treat only the case where (ii) or (iii) holds and $P'$ is defined as in \eqref{P'}, the remainder case is completely analogous.

We claim that $P'$ is a $1$-Lipschitz map preserving the distinguished point. The latter is clear, we show that it is $1$-Lipschitz. Pick $g,f\in G$, we have $$\|P'(g)-P'(f)\|=\|\int_H \delta(g\cdot h)-\delta(f\cdot h)d\mu(h)\|\leq \int_H \|\delta(g\cdot h)-\delta(f\cdot h)\|d\mu(h)=d(g,f),$$ where in the last equality we used that $\mu$ is probability and $d$ is invariant. It follows that $P'$ extends to a norm one linear operator $P:\F(G)\rightarrow \F(G)$. We claim it is the desired projection.

First we show that it is a projection. For every $h\in H$ let $P'_h:G\rightarrow \F(G)$ be the map defined for every $g\in G$ by $P'_h(g):=\delta(g\cdot h)-\delta(h)$. The following are straightforward to verify:
\begin{itemize}
    \item $P'_h$ is an isometry  with $P'_h(1)=0$, thus it extends to a norm one linear map $P_h:\F(G)\rightarrow \F(G)$.
    \item For every $g\in G$ we have $P'(g)=\int_H P'_h(g)d\mu(h)$ and so also for every $x\in\F(G)$ we have $P(x)=\int_H P_h(x)d\mu(h)$.
\end{itemize}
It follows that in order to show that $P^2=P$ it suffices to check that for every $g\in G$ and $h\in H$ we have $P_h\circ P(\delta(g))=P(\delta(g))$. Indeed, the previous equality implies $$P^2(\delta(g))=\int_H P_h\circ P(\delta(g))d\mu(h)=\int_H P(\delta(g))d\mu(h)=P(\delta(g)).$$ Since the set $\{x\in\F(G)\colon P^2(x)=P(x)\}$ is a closed linear subspace, we get that $P^2(x)=P(x)$ for all $x\in\F(G)$ since $\F(G)$ is the closed linear span of $\{\delta(g)\colon g\in G\}$. Let us thus fix $g\in G$ and $h\in H$. We have 
\begin{align*}
P_h\circ P(\delta(g))&=P_h\circ\int_H \delta(g\cdot f)-\delta(f)d\mu(f)\\
&=\int_H \Big(\delta(g\cdot f\cdot h)-\delta(h)\Big)-\Big(\delta(f\cdot h)-\delta(h)\Big)d\mu(f)\\
&=\int_H \delta(g\cdot f\cdot h)-\delta(f\cdot h)d\mu(f\cdot h)=P(\delta(g)),\end{align*} which finishes the claim.\medskip

Let $X$ denote the range of $P$ and let us show that it is linearly isometric with $\F(G/H,D)$. We define a map $T':G/H\rightarrow \F(G)$ by setting for any left coset $gH$ $$T'(gH):=P'(g).$$ We check that it is correctly defined and $1$-Lipschitz. For the former, we need to check that for any $g\in G$ and $h\in H$ we have $P'(g)=P'(gh)$, i.e. $\int_H \delta(g\cdot f)d\mu(f)=\int_H \delta(gh\cdot f)d\mu(f)$. But the equality follows from the invariance of $\mu$. To check that $T'$ is $1$-Lipschitz, pick two cosets $g_1H$ and $g_2H$ and suppose that $f\in H$ is such that $D(g_1H,g_2H)=d(g_1,g_2f)$. Then we have
\begin{align*}\|T'(g_1H)-T'(g_2H)\|&=\|\int_H \delta(g_1\cdot h)-\delta(g_2 f\cdot h)d\mu(h)\| \\
&=\int_H D(g_1H,g_2H)d\mu(f)=D(g_1H,g_2H),\end{align*}
showing that $T'$ is actually isometric. It follows that $T'$ extends to a norm one linear surjection $T:\F(G/H)\rightarrow X$. In order to show that $T$ is isometric, it suffices to prove that for any finite linear combination $x=\sum_i \alpha_i \delta(g_i H)$ we have $\|x\|_{\F(G/H)}=\|T(x)\|_{\F(G)}$. One inequality already follows from the fact that $\|T\|=1$, so we only need to prove $\|x\|_{\F(G/H)}\leq\|T(x)\|_{\F(G)}$. Let $f\in\rm{Lip}_0(G/H)$ be a $1$-Lipschitz function satisfying $\|x\|_{\F(G/H)}=|\sum_i \alpha_i f(g_iH)|$. Let $\tilde f$ denote its lift to $G$. That is, for any $g\in G$ and $h\in H$, $\tilde f(gh)=f(gH)$. It is clear that $\tilde f$ is $1$-Lipschitz. In the following, we shall not notationally distinguish between $\rm{Lip}_0(G/H)$-functions and their unique extension to linear functionals.

Since we have $\|T(x)\|\geq \tilde f(T(x))$, it suffices to check that $\tilde f(T(x))=f(x)$. For that, in turn, it suffices to check that for every $gH\in G/H$ we have $f(gH)=\tilde f(T'(g))$. We have $$\tilde f(T'(g))=\langle \int_H \delta(g\cdot h)-\delta(h)d\mu(h),\tilde f\rangle=\int_H \tilde f(g\cdot h)-\tilde f(h)d\mu(h)=f(gH)-f(H)=f(gH),$$ which finishes the proof.
\end{proof}

\begin{theorem}\label{thm:MAPcompactgrp}
Let $G$ be a compact group with a compatible left-invariant metric $d$. Then $\F(G,d)$ has the MAP.
\end{theorem}
Before embarking on the proof, we state the following folklore fact, which we prove for the convenience of the reader.

\begin{fact}\label{fact2}
For every compatible left-invariant metric $d$ on a compact metrizable group $G$ there exists a compatible bi-invariant metric $D$ such that $d(g,h)\leq D(g,h)$ for all $g,h\in G$.
\end{fact}
\begin{proof}[Proof of Fact \ref{fact2}]
Let $d$ be an arbitrary compatible left-invariant metric on $G$. We define a compatible bi-invariant metric $D$ by setting, for any $g,f\in G$, $$D(g,f):=\max_{h\in G} d(gh,fh).$$ Clearly, $d\leq D$ and $D$ is bi-invariant. Let us check that $D$ satisfies the triangle inequality and it is compatible with the topology of $G$.

For the triangle inequality, pick $g,h,f\in G$ and let us show that $D(g,f)\leq D(g,h)+D(h,f)$. Let $x\in G$ be such that $D(g,f)=d(gx,fx).$ Then we have $$D(g,f)=d(gx,fx)\leq d(gx,hx)+d(hx,fx)\leq D(g,h)+D(h,f),$$ showing the traingle inequality.

In order to show that it is compatible with the topology, using left-invariance, it suffices to show that for every sequence $(g_n)_n\subseteq G$ we have $d(g_n,1)\to 0$ if and only if $D(g_n,1)\to 0$. One direction follows immediately from the fact that $d\leq D$. Thus we only need to show that if $d(g_n,1)\to 0$, then $D(g_n,1)\to 0$. Suppose it is not the case and assume without loss of generality that $\lim_n D(g_n,1)=r>0$. For each $n$, let $h_n\in G$ be such that $D(g_n,1)=d(g_nh_n,h_n)$. Again, without loss of generality, we may assume that $(h_n)_n$ converges to some $h\in G$. Then, using that $(g_n)_n$ converges to $1$ since $d$ is compatible, we have $$r=\lim_n D(g_n,1)=\lim_n d(g_nh_n,h_n)=d(h,h)=0,$$ a contradiction finishing the proof.
\end{proof}

\begin{proof}[Proof of Theorem~\ref{thm:MAPcompactgrp}]
Let $G$ be a compact metrizable group with compatible left-invariant metric $d$. We shall also fix some compatible bi-invariant metric $D$ on $G$ which majorizes $d$, i.e. $d\leq D$, which exists by Fact~\ref{fact2}. It is a standard consequence of Peter-Weyl's theorem (\cite[Theorem 4.20]{Knapp}) that $G$ can be topologically embedded into an infinite direct product $\prod_{i\in\mathbb{N}} U_i$, where each $U_i$ is a finite-dimensional unitary group. Indeed, by \cite[Theorem 4.20]{Knapp} finite-dimensional unitary representations of $G$ separate points. Since $G$ is separable, one can find a sequence $\{\pi_n:G\rightarrow B(H_n)\}_{n\in\N}$ of finite dimensional unitary representations separating the points, and their product $\prod_{n\in\N} \pi_n$ is then, using also compactness of $G$, a topological embedding of $G$ into a countable product of unitary groups. In particular, $G$ is an inverse limit of a sequence of compact Lie groups $(G_n)_n$. Indeed, let $\Psi:=\prod_{n\in\N} \pi_n: G\rightarrow \prod_{i\in\mathbb{N}} U_i$ be the embedding and let, for each $n\in\mathbb{N}$, $P_n:\prod_{i\in\mathbb{N}} U_i\rightarrow \prod_{i\leq n} U_i$ be the projection on the first $n$-coordinates. Then for each $n\in\mathbb{N}$, $G_n:=P_n\circ \Psi[G]$ is a compact Lie group (a closed subgroup of $\prod_{i\leq n} U_i$), and $(G_n)_n$ form an inverse system whose limit is $G$. It follows that there exists a decreasing sequence of compact normal subgroups $(H_n)_n$ of $G$ such that $\bigcap_n H_n=\{1\}$, and for every $n\in\mathbb{N}$, $G/H_n=G_n$. For each $n$, we denote by $d_n$ the quotient metric on $G_n$, which is then compatible and left-invariant on $G_n$. In the following, $\F(G)$ is meant to be $\F(G,d)$ and $\F(G_n)$ is meant to be $\F(G_n,d_n)$, for all $n$.

Notice that in the following claim we rather work with the bi-invariant metric $D$ majorizing $d$, established in Fact \ref{fact2}.\\

{\bf Claim 1.} $\rm{diam}_D(H_n)\to 0$, where $\rm{diam}_D(H_n):=\sup_{g,h\in H_n} D(g,h)=\sup_{g\in H_n} D(g,1)$.
\begin{proof}[Proof of Claim 1.]
Suppose on the contrary that there exist $\varepsilon>0$ and a sequence $(h_n)_n\subseteq G$ such that $h_n\in H_n$ and $D(h_n,1)\geq \varepsilon$. Since $G$ is compact, the sequence without loss of generality converges to some $h\in G$. By the continuity of the metric, we have $D(h,1)\geq \varepsilon$, thus in particular $h\neq 1$. On the other hand, since the sequence $(H_n)_n$ is decreasing and each of the subgroups is closed, $h\in \bigcap_n H_n$. This contradicts that $\bigcap_n H_n=\{1\}$.
\end{proof}

For each $n$, let $\mu_n$ be the normalized invariant Haar measure on the compact subgroup $H_n$, and let $P'_n$, resp. $P_n$ be the map, resp. projection from Proposition~\ref{prop:projection} applied to the groups $G$ and $H_n$ with the metric $d$ and the quotient metric $d_n$ on $G/H_n$. \medskip

{\bf Claim 2.} For every $x\in \F(G)$, we have $x=\lim_{n\to\infty} P_n(x)$. 
\begin{proof}[Proof of Claim 2.]
Suppose first  that $x\in \F(G)$ is a finite linear combination of Dirac elements. That is, there are $m\in \N$, $g_1,\dots,g_m \in G$ and $\alpha_1,\dots,\alpha_m\in\R$ with $x=\sum_{i=1}^m \alpha_i \delta(g_i)$. For a fixed $n$, let us compute $\|x-P_n(x)\|$. We have 
\begin{align*}\|x-P_n(x)\|& = \|x-\int_{H_n} h\cdot x - (\sum_{i=1}^m \alpha_i)\delta(h)\,d\mu_n(h)\|\\
&\leq \int_{H_n}\|x- h\cdot x \|\,d\mu_n(h) + (\sum_{i=1}^m |\alpha_i|)\|\int_{H_n}\delta(h)\,d\mu_n(h)\|.
\end{align*}
Suppose that $\rm{diam}_D(H_n)= \varepsilon_n$. Then for every $h\in H_n$ we have 
\begin{align*}
\|x-h\cdot x\|& =\|\sum_{i=1}^m \alpha_i \delta(g_i)-\sum_{i=1}^m \alpha_i \delta(h\cdot g_i)\|\leq \sum_{i=1}^m |\alpha_i| d(h\cdot g_i,g_i)\\
& \leq \sum_{i=1}^m |\alpha_i| D(h\cdot g_i,g_i)\leq \sum_{i=1}^m |\alpha_i| D(h,1)\leq \sum_{i=1}^m |\alpha_i| \varepsilon_n,
\end{align*}
and it follows that 
$\int_{H_n}\|x-h\cdot x\|d\mu_n(h)\leq   \sum_{i=1}^m |\alpha_i| \varepsilon_n$. 
On the other hand, for each $f\in B_{\rm{Lip}_0(G)}$, 
$$
|\langle \int_{H_n}\delta(h)\,d\mu_n(h),f \rangle| = |\int_{H_n}f(h)\,d\mu_n(h)|\leq \int_{H_n}\rm{diam}_d(H_n)\,d\mu_n(h)\leq \int_{H_n} \rm{diam}_D(H_n)\leq \varepsilon_n,
$$
so $\|\int_{H_n}\delta(h)\,d\mu_n(h)\|\leq \varepsilon_n$. 
It follows that $$\|x-P_n(x)\|\leq 2\sum_{i=1}^m |\alpha_i|\varepsilon_n,$$ so by \textbf{Claim 1}, $\lim_{n\rightarrow\infty} P_n(x)=x$.\\

Now let  $x\in \F(G)$ be an arbitrary element and note that, for each finitely supported $y$,
\begin{align*}\|x-P_n(x)\|&\leq \|x-y\|+\|y-P_n(y)\| + \|P_n(y) - P_n(x)\|\\
& = \|x-y\| +\|y-P_n(y)\| + \|P_n\|\|y - x\|=2\|x-y\| +\|y-P_n(y)\|.
\end{align*}
Hence, 
$$
\limsup_{n\rightarrow\infty}\|x-P_n(x)\|\leq 2\|x-y\| +\lim_{n\rightarrow\infty}\|y-P_n(y)\| = 2\|x-y\|. $$
Since $y$ can be chosen arbitrarily close to $x$, the result follows. 

\end{proof}

We are ready to show that $\F(G)$ has MAP. Pick finitely many $x_1,\ldots,x_m\in \F(G)$ and some $\varepsilon$. By the previous claim we can find $n$ so that for all $i\leq m$ we have $\|x_i-P_n(x_i)\|<\varepsilon/2$. By Theorem~\ref{proptorus}, $F(G_n)$ has the MAP. Thus there exists a norm one finite rank operator $T':F(G_n)\rightarrow F(G_n)$ such that for all $i\leq m$ we have $\|P_n(x_i)-T'\circ P_n(x_i)\|<\varepsilon/2$. Set $T:=T'\circ P_n$. It is a norm one finite rank operator from $F(G)$ into $F(G_n)\subseteq F(G)$ such that for all $i\leq n$ we have $$\|T(x_i)-x_i\|\leq \|x_i-P_n(x_i)\|+\|P_n(x_i)-T(x_i)\|<\varepsilon/2+\varepsilon/2=\varepsilon.$$
This finishes the proof.
\end{proof}

Finally, we show the metric approximation property also for Lipschitz-free spaces over homogeneous spaces of compact metrizable groups. We recall that a \emph{homogeneous space for a group} $G$ is a topological space on which $G$ acts transitively. It can be identified with the left coset space $G/H$, where $H$ is a subgroup, with the quotient topology. To ensure some regularity, we must restrict to compatible metrics which are $G$-invariant, i.e. those so that the action of $G$ is by isometries. 
\begin{theorem}\label{thm:homogeneousspace}
Let $M$ be a $G$-homogeneous space, where $G$ is a compact group. Suppose that $D$ is a $G$-invariant metric on $M$ that is a quotient of some bi-invariant metric $d$ on $G$. Then $\F(M,D)$ has the MAP.
\end{theorem}
\begin{proof}
We have that $M$ is isomorphic to $G/H$, where $H$ is a stabilizer of some point $0\in M$. $H$ is then a closed subgroup. By Theorem~\ref{thm:MAPcompactgrp}, $\F(G,d)$ has the MAP. By Proposition~\ref{prop:projection}, there exists a norm one projection $P:\F(G)\rightarrow X$, where $X$ is linearly isometric to $\F(G/H,D)$. Since MAP is inherited by $1$-complemented subspaces, we are done.
\end{proof}

\section{Lipschitz-free spaces over finitely generated groups}\label{section:fingengrps}
The goal of this section is to prove Theorem~\ref{thm:intro2}, show examples of situations to which the theorem applies, and provide applications. We mention that the class of Lipschitz-free spaces for which it is known they have the Schauder basis is still rather limited. It is proved in \cite{HaPe} that $\F(\R^n)$ and $\F(\ell_1)$ have a Schauder basis, in \cite{CuDo} that free spaces over any separable ultrametric space have a monotone Schauder basis, and in \cite{hajek2017some} that $\F(N)$ has a Schauder basis if $N$ is a net in a separable $C(K)$-space. Obviously, it also follows that free spaces isomorphic to these Banach spaces have basis as well, so e.g. by \cite{kaufmann2015products}, $\F(B_{\ell_1})$ and $\F(B_{\R^n})$ have bases.\bigskip

We start by recalling the fundamental idea of geometric group theory - how to view finitely generated groups as metric spaces. Our standard reference for geometric group theory is \cite{DruKap}. In contrast to the case of compact metrizable groups, we will not consider arbitrary compatible left-invariant metrics on such groups, just certain canonical and maximal, in a sense, ones, called word metrics.

Let $G$ be a finitely generated group. Let $S\subseteq G$ be a finite symmetric generating subset (recall that `symmetric' means that for each $s\in S$, also $s^{-1}\in S$). Recall that we can then define a (left-invariant) metric $d_S$, called \emph{word metric}, on $G$ by defining, for $g\neq h\in G$, $$d_S(g,h):=\min\{n\in\Nat\colon \exists s_1,\ldots,s_n\in S\; (g=hs_1\ldots s_n)\}.$$

It is well known and easy to check that by chosing another finite symmetric generating set $T\subseteq G$, the identity map between $(G,d_S)$ and $(G,d_T)$ is bilipschitz. In particular, the isomorphism class of the Banach space $\F(G)$ is well-defined.\medskip

Since every finitely generated group with its word metric is a countable proper metric space, it immediately follows from \cite{dalet2015free} that $\F(G)$ has the MAP. Therefore we will aim for stronger properties and indeed we shall present a class of finitely generated group such that free spaces over any of them has the Schauder basis.\medskip

Fix now some finitely generated group $G$ and a finite symmetric generating set $S$. Next choose arbitrarily a linear order on the set $S$. Consider now $S$ as an alphabet, i.e. its elements are considered to be letters, and denote by $W$ the set of all \emph{reduced words} over the alphabet $S$. That is, each element $w$ of $W$ is a string of symbols $s_1 s_2\ldots s_n$ from $S$ such that 
for no $i<n$, the letters $s_i$ and $s_{i+1}$ are inverses of each other when viewed as group elements. For each $w\in W$,
\begin{itemize}
\item by $|w|$ we denote the length of the word, i.e. the number of its letters; 
\item by $w_G$ we denote the corresponding group element of $G$, i.e. we evaluate the letters of $w$ as elements of $G$;
\item for every $i<|w|$, by $w(i)$, we denote the $i$-th letter of $w$, and by $w(\leq i)$, we denote the word obtained from $w$ by taking the first $i$ letters.
\end{itemize}
The fixed linear order on the set $S$ defines a lexicographical order on $W$ which we shall denote by $\preceq'$. We define another linear order $\preceq$ on $W$, called the \emph{shortlex} order, by setting, for $w,v\in W$, $$w\preceq v \text{ if either }|w|<|v|, \text{ or }|w|=|v|\text{ and }w\preceq' v.$$

For an element $g\in G$, by $W_g$ we denote the set $\{w\in W\colon w_G=g\}$ and by $w_g$ the minimal element of the set $W_g$, which is easily verified to exist, in $\preceq$. If there is no danger of confusion, for an element $g\in G$ we shall denote by $|g|$ the number $|w_g|$ which is equal to $d_S(g,1_G)$.

Finally, we use the linear order $\preceq$ on $W$ to define a linear order $\leq$ on $G$. For $g,h\in G$ we set $$g\leq h\text{ if }w_g\preceq w_h.$$

\begin{definition}
We call $G$ \emph{shortlex combable}, with respect to a fixed symmetric generating set $S$ and a linear order on $S$, if there exists a constant $K\geq 1$ such that for every $g,h\in G$ with $d_S(g,h)=1$ and for every $i\leq \min\{d_S(g,1_G),d_S(h,1_G)\}$ we have $$d_S((w_g(\leq i))_G, (w_h(\leq i))_G)\leq K.$$

The constant $K$ will be called  the \emph{combability constant} of $G$.
\end{definition}
First we show how such groups are useful for our purposes. Then we provide examples and show some applications. The following is the main result of this section.

\begin{theorem}\label{thm:shortlex}
Let $G$ be a finitely generated shortlex combable group (with respect to $S$ equipped with some linear order). Then $\F(G,d_S)$ has a Schauder basis.
\end{theorem}
\begin{proof}
Since the linear order $\leq$ on $G$  is clearly isomorphic to the standard order on $\Nat$, we can use it to enumerate $G$ as $(g_n)_{n\in\Nat}$. For each $n\in\Nat$, set $G_n:=\{g_i\colon i\leq n\}$. For each $n$ we now define maps $P_n: G\rightarrow G_n$ as follows. First, set $m=\max\{|h|\colon h\in G_n\}$, then for $g\in G$, set $$P_n(g):=\begin{cases} g & \text{if }g\in G_n\\
(w_g(\leq m))_G & \text{if }g\notin G_n, (w_g(\leq m))_G\in G_n\\
(w_g(\leq m-1))_G & \text{otherwise}.
\end{cases}$$
We leave to the reader the straightforward verification that $P_n$ is well defined.\\

\noindent{\bf Claim.} The maps $(P_n)_{n\in\Nat}$ are uniformly bounded Lipschitz commutting retractions.\\

First we check that for every $n,m\in\Nat$, $P_n\circ P_m=P_m\circ P_n$, i.e. the maps commute. For every $g\in G$ and $0\leq i\leq |g|$, denote by $g_i$ the element $(w_g(\leq i))_G$. That is, $g_0=1_G$, $g_{|g|}=g$, and $g_0,g_1,\ldots,g_{|g|}$ is a geodesic path in $G$ from the identity element $1_G$ to $g$. It is easy to see that for each $n\in\Nat$, $P_n(g)=g_i$, where $i$ is the largest integer so that $g_i\in G_n$. With this observation it is now clear that the maps $(P_n)_n$ commute. 

Let $K$ be the combability constant of $G$. We show that each $P_n$ is a $K+1$-Lipschitz retraction on its image. It is obviously a retraction, so it suffices to show that for every $n\in\Nat$ and $g,h\in G$ with $d_S(g,h)=1$ we have $d_S(P_n(g),P_n(h))\leq K+1$. Let such $n$ and $g,h\in G$ be fixed. We distinguish two cases.\\

\noindent{\it Case 1.} At least one of $g,h$ lies in $G_n$: Say that $g\in G_n$. If also $h\in G_n$, then there is nothing to prove since $P_n(g)=g$ and $P_n(h)=h$. So suppose that $h\notin G_n$. Necessarily we have $h_{|h|-1}\leq g$ since otherwise the path $1_G, g_1,\ldots, g, h$ would be a geodesic path from $1_G$ to $h$ of length $|h|$ smaller in the lexicographical ordering than the path $1_G, h_1,\ldots,h_{|h|-1},h$. It follows that $h_{|h|-1}\in G_n$, so $P_n(h)=h_{|h|-1}$. Since $P_n(g)=g$ we get $$d_S(P_n(g),P_n(h))\leq 2.$$\\

\noindent{\it Case 2.} We have $g,h\notin G_n$. Note that then $P_n(g)=g_i$ and $P_n(h)=h_j$, for some $i,j< \max\{|g|,|h|\}$, where $|i-j|\leq 1$. Since by the definition of shortlex combability, $d_S(g_i,h_i)\leq K$ and $d_S(g_j,h_j)\leq K$, we get that $$d_S(P_n(g),P_n(h))=d_S(g_i,h_j)\leq K+1.$$

This finishes the proof of the claim.\\

Finally, for each $n\in\Nat$ we denote by $L_n: \F(G)\rightarrow \F(G_n)$ the lift of $P_n$, the unique linear operator extending $P_n$. The properties of $(P_n)_n$ imply that
\begin{itemize}
	\item for each $n\in\Nat$, the map $L_n$ is a linear projection onto a finite-dimensional subspace of norm bounded by $K+1$;
	\item the projections $(L_n)_n$ commute;
	\item the dimension of the range $L_n[\F(G)]$ is the dimension of $\F(G_n)$, which is equal to $n$.
\end{itemize}
Since now obviously for every $x\in \F(G)$ (notice that it suffices to verify it for the dense set $\mathrm{span}\{\delta_g\colon g\in G\}$) we have $$\lim_{n\to\infty} L_n(x)=x,$$ then by \cite[Proposition 1.1.7]{AlbiacKalton} $\F(G)$ possesses a Schauder basis.

\end{proof}

\subsection{Examples}\label{subsec:examples}\hfill\bigskip

\noindent {\it Finitely generated abelian groups}. Recall that every infinite finitely generated abelian group $A$ is of the form $\Int^n\oplus F$, where $n\geq 1$ and $F=\{0,f_1,\ldots,f_m\}$ is a finite abelian group. Let $e_1,\ldots,e_n$ be the canonical generators of $\Int^n$. Then $e_1\leq -e_1\leq e_2\leq -e_2\leq\ldots\leq e_n\leq -e_n\leq f_1\leq \ldots\leq f_m$ is a linearly ordered finite symmetric generating set. We leave to the reader to verify that with this ordered generating set $A$ is shortlex combable.\\

\noindent {\it Free groups}. Let $n\geq 1$ and let $F_n$ be a free group on $n$ generators (which is $\Int$ for $n=1$). Let $a_1,\ldots,a_n$ be the free generators. It is again an exercise that with the order $a_1\leq a^{-1}_1\leq\ldots\leq a_n\leq a_n^{-1}$, the group $F_n$ is shortlex combable. We note however, that $\F(F_n)$ is linearly isometric to $\ell_1$, thus admits monotone Schauder basis. Indeed, this can be verified directly by noticing that the set $\{\delta(g)-\delta(h)\in\F(F_n)\colon d(g,h)=1, d(g,1)>d(h,1)\}$ is equivalent to the standard basis of $\ell_1$.\\ 

\noindent {\it Hyperbolic groups}. Recall that a geodesic metric space $(M,d)$ is (Rips)-hyperbolic if there exists a constant $K\geq 0$, \emph{hyperbolicity constant}, such that for any triple of points $x,y,z\in M$ and geodesic segments $S_1,S_2,S_3$ connecting each pair of the triple we have $d_{H}(S_i,S_j\cup S_k)\leq K$, where $d_H$ is the Hausdorff distance and $i,j,k$ are pairwise different from $\{1,2,3\}$. In other words, for each $i\leq 3$ and each point $x\in S_i$, $d(x,S_j\cup S_k)\leq K$, where $S_j$ and $S_k$ are the other geodesics besides $S_i$.
	
The notion of hyperbolicity makes sense even for metric spaces which are not literally geodesic, but when a reasonable notion of geodesic segment can be defined. This is the case e.g. for finitely generated groups with word metrics, where a geodesic segment between elements $x,y\in G$ is a sequence $g_0=x,\ldots,g_n=y$, where $n=d(x,y)$, and $d(g_i,g_{i+1})=1$, for $i<n$.
\medskip
	
Let $G$ be a finitely generated hyperbolic group (with hyperbolicity constant $K$), generated by elements $a_1,\ldots,a_n$. We claim that $G$ with the ordered generating set $a_1\leq a^{-1}_1\leq\ldots\leq a_n\leq a_n^{-1}$ is shortlex combable. Indeed, pick $g,h\in G$ with $d(g,h)=1$, and $i<\max\{|g|,|h|\}$. We show that $d(g_i,h_i)\leq 2K+2$. We have two cases.
\begin{enumerate}
	\item Either $g_i=g$ or $h_i=h$ (or both). Then it is clear that $d(g_i,h_i)\leq 2$.
	\item We have $g_i\neq g$ and $h_i\neq h$. There are geodesic segments $g_0=1_G,\ldots,g_i,\ldots,g$, $h_0=1_G,\ldots,h_i,\ldots,h$, and $g,h$ (of length $1$). We consider thr triple of points $1_G,x,y$ and the geodesic segments between them as above. By definition of hyperbolicity with constant $K$, there exists point $z\in\{1_G=h_0,\ldots, h,g\}$ such that $d(g_i,z)\leq K$. Assume first that $z=h_j$, for some $j\leq |h|$.
	\begin{itemize}
		\item $j\geq i$: Since $d(g_i,h_j)\leq K$, by triangle inequality we get that $j\leq i+K$, therefore $d(g_i,h_i)\leq 2K$.
		\item $j<i$: Again by triangle inequality we get that $j\geq i-K$, so $d(g_i,h_i)\leq 2K$.
	\end{itemize}
	Finally, if $z=g$, then $d(g_i,h)\leq K+1$, so again by triangle inequality we get $|h|\leq i+K+1$, so $d(g_i,h_i)\leq 2(K+1)$.\\
\end{enumerate}

\noindent {\it Large-type Artin groups.} Holt and Rees in \cite{HoRe} prove that large-type Artin groups are shortlex automatic which immediately from the definition implies being shortlex combable. Artin groups in general belong to one of the most studied classes of groups in geometric group theory. We refer the reader to \cite{HoRe} for the notion of shortlex automaticity and for the definition of large-type Artin groups.

We do not know whether there are groups that are shortlex combable but not shortlex automatic. We refer to \cite{HoRe} for details.

\subsection{Applications}\hfill\bigskip

Our main goal is to prove that for any net $N$ in a real hyperbolic $n$-space $\Hy^n$ (whose definition we recall later), we have that $\F(N)$ has the Schauder basis. This will be an immediate consequence of Theorem~\ref{thm:shortlex} and several standard more general results that we present below.

We recall that an action of a group $G$ on a metric space $X$ by isometries is \emph{free} if for every $g\in G$ and $x\in X$, $g\cdot x\neq x$, and \emph{cocompact} if there exists a compact set $K\subseteq X$ such that $\bigcup_{g\in G} g\cdot K=X$. These actions, or rather more generally properly discontinuous cocompact actions (see \cite[Chapter 5]{DruKap}), are one of the most studied in geometric group theory.
\begin{proposition}\label{prop:actinggroup}
Let $G$ be a finitely generated group acting freely and cocompactly on a proper geodesic metric space $X$ by isometries. Then for every net $N\subseteq X$, we have $\F(N)\simeq \F(G)$.
\end{proposition}
\begin{proof}
Let $G$, $X$ and an action of $G$ on $X$ as in the statement of the proposition be fixed. First we invoke \cite[Proposition 5]{hajek2017some} which says that for any two nets $N_1,N_2\subseteq X$ we have $\F(N_1)\simeq\F(N_2)$. It follows that it suffices to find one net $N\subseteq X$ such that $\F(N)\simeq\F(G)$. Let $0\in X$ be a distinguished point and let $N$ be the $G$-orbit of $0$. Since $X$ is proper and the action is cocompact, it follows that $N$ is a net in $X$. To show that $\F(N)\simeq\F(G)$ it suffices to prove that $N$ with the restriction of the metric on $X$ is bi-Lipschitz to $G$ (with its word metric). By the Milnor-Schwarz lemma (see \cite[Theorem 8.37]{DruKap}), the map $\phi:G\rightarrow N$ defined by $\phi(g):=g\cdot 0$ is a quasi-isometry (we refer reader not familiar with quasi-isometries again to \cite{DruKap}). Since the action is free, it is also a bijection. It is easy to check that a bijective quasi-isometry between two uniformly discrete metric spaces is in fact a bi-Lipschitz equivalence. This finishes the proof.
\end{proof}

We now recall the definition of the real hyperbolic $n$-space. There are many definitions and we refer the reader to \cite[Chapter 4]{DruKap} for a more thorough discussion. We define $\Hy^n$, for $n\geq 2$, as follows. First we define the following quadratic form; for $x,y\in\Rea^{n+1}$: $$\langle x,y\rangle:=\sum_{i=1}^n x_iy_i-x_{n+1}y_{n+1},$$ and we set $$\Hy^n:=\{x\in\Rea^{n+1}\colon \langle x,x\rangle=-1,x_{n+1}>0\}.$$
A metric $d$ on $\Hy^n$ can be defined using the formula, for $x,y\in\Hy^n$, $$\cosh d(x,y)=-\langle x,y\rangle.$$
We now state the main result of this subsection.
\begin{corollary}\label{cor:hyperbolicnet}
Let $n\geq 2$ and let $N\subseteq \Hy^n$ be a net. Then $\F(N)$ has a Schauder basis.
\end{corollary}
\begin{proof}
	It suffices to find a group $G$ acting freely and cocompactly on $\Hy^n$. Indeed, again by the Milnor-Schwarz lemma (\cite[Theorem 8.37]{DruKap}), such $G$ is then finitely generated and quasi-isometric to $\Hy^n$. It follows that $G$ is hyperbolic (\cite[Observation 11.125]{DruKap}). Therefore, as we demonstrated in Subsection~\ref{subsec:examples}, $G$ is shortlex combable. Applying Theorem~\ref{thm:shortlex}, we get that $\F(G)$ has a Schauder basis. Finally, an application of Proposition~\ref{prop:actinggroup} finishes the argument.\medskip

In order to find a group acting freely and cocompactly on $\Hy^n$, one can use several standard results from Riemannian geometry. Let $M$ be an arbitrary $n$-dimensional compact Riemannian manifold without boundary of constant sectional curvature $-1$ equipped with some Riemannian metric. By \cite[Theorem 3.32]{BrHa}, its universal cover (refer to any standard textbook on algebraic topology, e.g. \cite{Hatch}) is isometric to $\Hy^n$. Another standard argument from algebraic topology (see e.g. \cite[Proposition 1.39]{Hatch}) shows that the fundamental group $\pi_1(M)$ acts on the universal cover $\Hy^n$ by deck transformations, which is a free cocompact action by isometries. This finishes the proof.
\end{proof}
We remark that in contrast to the Euclidean space $\R^n$, in which there exist two non Lipschitz equivalent nets (\cite{BuKl}), it has been proved in \cite{Bog} that all nets in $\Hy^n$ are bi-Lipschitz equivalent.

\begin{remark}
Let $\mathcal{G}$ denote the set of all finitely generated hyperbolic groups. A fair question is how many isomorphism types of Banach spaces the set $\{\F(G)\colon G\in\mathcal{G}\}$ contains. As we mentioned, since the free group $F_n$ is hyperbolic, it contains $\F(F_n)\simeq \ell_1$. It is unknown to us whether or not there exists $G\in\mathcal{G}$ such that $\F(G)\not\simeq \ell_1$. Thus it is relevant at this point to reiterate \cite[Question 1]{candido2019isomorphisms}, which asks precisely about an example of $G\in\mathcal{G}$ with $\F(G)\not\simeq\ell_1$, and the discussion following the question. Either answer would bring interesting consequences. If there are such $G$, then we have, potentially many, new examples of Lipschitz-free spaces with Schauder basis. If on the other hand for every $G\in\mathcal{G}$, $\F(G)\simeq \ell_1$, then we have an example of a group with Kazhdan's property (T) that has a metrically proper action by isometries on a renorming of $\ell_1$. We refer to \cite[Question 1]{candido2019isomorphisms} where this is properly discussed and the importance of such a result is explained.
\end{remark}

\section{Problems and Notes}\label{section:problems}

Let us finish by posing some natural questions that follow up this work. The first is whether or not we can generalize Theorem \ref{thm:intro1} to locally compact groups: \medskip

\begin{question} Let $G$ be a locally compact group equipped with a compatible and left-invariant metric. Does $\F(G)$ have the MAP?
\end{question}

In the noncompact case, we have positive answer for finite dimensional Banach spaces with their norm induced metrics. Although this would be a consequence of the mentioned result from  \cite{godefroy2003lipschitz} which states that a Banach space has $\lambda$-BAP if and only if $\F(X)$ does, actually the order of the proofs is reversed. First,  Godefroy and Kalton prove that, for finite dimensional Banach spaces $X$, $\F(X)$ has the MAP, and then use this to prove the result for general Banach spaces. The proof of the finite dimensional part also involves harmonic averaging. 

Since Lipschitz-free spaces over finite dimensional Banach spaces even possess a Schauder basis (\cite{HaPe}), we suggest to turn the attention to locally compact metric groups that are very closely related to such Banach spaces; that is, Carnot groups. We note that a Carnot group is both analytically and algebraically a mild generalization of a finite dimensional Banach space and many results from geometric analysis on Euclidean spaces have been generalized to the setting of Carnot groups (see \cite{LeD}). The duals of Lipschitz-free spaces over Carnot groups have been already investigated in \cite{candido2019isomorphisms} and their Lipschitz-free spaces in \cite{albiac2020lipschitz}.

\begin{question}
Let $G$ be a Carnot group with a Carnot-Carath\' eodory metric. Does $\F(G)$ has the Schauder basis?
\end{question}
In the case of connected Lie groups, one has a canonical compatible left-invariant metric: the left-invariant Riemannian metric. Since such a metric is locally bi-Lipschitz equivalent to the Euclidean metric, and isomorphic properties of Lipschitz-free spaces often depend only on the local behaviour of the metric, the following question is very natural. We note that one could even replace Lie group there with a general connected Riemannian manifold (without boundary).
\begin{question}\label{quest:liegrp}
Let $G$ be a connected (real) Lie group equipped with a left-invariant Riemannian metric. Do we have $\F(G)\simeq \F(\R^n)$, where $n$ is the dimension of the (real) Lie algebra $\mathfrak{g}$ of $G$?
\end{question}

Still in the compact setting, we note that we only required in Theorem \ref{thm:homogeneousspace} that the metric in $G$-homogeneous space is a quotient of a metric in $G$ so that we could apply directly Proposition \ref{prop:projection}. So one could ask if we can drop this condition.

\begin{question} Let $G$ be a (locally) compact group, $M$ be a compact $G$-homogeneous space, and $d$ be a compatible $G$-invariant metric on $M$. Does $\F(G)$ have the MAP? \end{question}

In Appendix A, the reader will find a positive answer in the specific case of the sphere $\Sph^{n-1}=O(n)/O(n-1)$. \medskip

It is natural also to ask about generalizations of Theorem \ref{thm:intro2}.

\begin{question}
Let $G$ be a finitely generated group equipped with a word metric. Does $\F(G)$ admit a Schauder basis, or at least a finite dimensional decomposition?
\end{question}
In Corollary~\ref{cor:hyperbolicnet}, we have shown that for any net $N\subseteq \Hy^n$, $\F(N)$ has a Schauder basis. Given the prominence of the space $\Hy^n$ in geometry and beyond, it is important to understand the Lipschitz-free space of $\Hy^n$ itself. We also note that the hyperbolic spaces $\Hy^n$ together with the Euclidean spaces $\R^n$ and Euclidean spheres $\Sph^n$ are important as the model spaces of spaces of constant curvature (see e.g. \cite{BrHa} for a thorough treatment). Since for $\R^n$ and $\Sph^n$ with their canonical Euclidean metrics we have by \cite[Theorem 4.21]{albiac2020lipschitz}, $\F(\R^n)\simeq\F(\Sph^n)$, the answer to the following would be desirable.
\begin{question}
Does $\F(\Hy^n)$ have the Schauder basis? Do we have $\F(\Hy^n)\simeq \F(\R^n)$?
\end{question}
We note that the previous question is related to the stronger version of Question~\ref{quest:liegrp} that considers a general Riemannian manifold since $\Hy^n$ is a Riemannian manifold with Riemannian metric (see \cite[Proposition 6.17]{BrHa}).

\appendix
\section{On the MAP for $\F(\Sph^n)$}\label{appendSphere}


Let $d\geq 2$ and let $\Sph^{d-1}=O(d)/O(d-1)$ be the $(d-1)$-dimensional sphere equipped with a rotation-invariant metric $D$ which is compatible with the usual topology. Here, we are not assuming that $D$ is a quotient metric, as the ones described in Subsection \ref{subsection:generalCpt}. Let us  show that  $\F(\Sph^{d-1},D)$ has the MAP.  

\begin{proof}
The proof will also follow from Proposition \ref{tool} and summability results. We recall the definitions and results from harmonic analysis that we will use, and establish again a Lipschitz version of Young's convolution inequality for our setting.  Equip  $\Sph^{d-1}$ with its surface area measure $\sigma$, and denote $\omega_d=\sigma(\Sph^{d-1})$. Convolution on $\Sph^{d-1}$ can be defined as follows. First let $\Lambda = (d-2)/2$, and consider the weighted $L^1$ space $L^1(w_\Lambda,[-1,1])$, where $w_\Lambda (x) = (1-x^2)^{\Lambda - 1/2}$ for each $x\in ]-1,1[$. Recall that the norm in $L^1(w_\Lambda,[-1,1])$ is defined by
$$
\|f\|_{\Lambda,1} = c_\Lambda \int_{-1}^1 |f(x)|w_\Lambda (x)dx,
$$
where $c_\Lambda$ is the normalization constant such that $c_\Lambda \int_{-1}^1w_\Lambda (x)dx =1$.
 For each  $f\in L^1(\Sph^{d-1})$ and $g\in L^1(w_\Lambda,[-1,1])$, the convolution $f\ast g: \Sph^{d-1}\rightarrow \R$ is defined by  
$$
(f\ast g)(x)= \frac{1}{\omega_d}\int_{\Sph^{d-1}} f(y)g(x\cdot y)d\sigma(y).
$$
We now prove the validity of Young's inequality. Let $f\in \mathrm{Lip}(\Sph^{d-1})$ and $g \in L^1(w_\Lambda,[-1,1])$, and let $x,y\in \Sph^{d-1}$. There is a rotation $R$ with $y=Rx$, thus
\begin{align*}
|f*g(x) - f*g(y)| & = \left| \frac{1}{\omega_d}\int_{\Sph^{d-1}} f(z)g(x\cdot z)d\sigma(z) - \frac{1}{\omega_n}\int_{\Sph^{d-1}} f(z)g(Rx\cdot z)d\sigma(z) \right|\\
& = \frac{1}{\omega_d}\left| \int_{\Sph^{d-1}} f(z)g(x\cdot z)d\sigma(z) - \int_{\Sph^{d-1}} f(Rz)g(Rx\cdot Rz)d\sigma(z) \right|\\
& = \frac{1}{\omega_d}\left| \int_{\Sph^{d-1}} f(z)g(x\cdot z)d\sigma(z) - \int_{\Sph^{d-1}} f(Rz)g(x\cdot z)d\sigma(z) \right|\\
& \leq \frac{1}{\omega_d}\int_{\Sph^{d-1}} |f(z)-f(Rz)||g(x\cdot z)|d\sigma(z) \\
& \leq \|f\|_{\mathrm{Lip}}\frac{1}{\omega_d}\int_{\Sph^{d-1}} D(z,Rz)|g(x\cdot z)|d\sigma(z)\\
& = \|f\|_{\mathrm{Lip}}D(x,y)\frac{1}{\omega_d}\int_{\Sph^{d-1}} |g(x\cdot z)|d\sigma(z)\\
& = \|f\|_{\mathrm{Lip}}D(x,y)\frac{1}{\omega_n}\int_{-1}^1 |g(t)|w_\Lambda (t) dt \leq \|f\|_{\mathrm{Lip}}\|g\|_{\Lambda,1}D(x,y).
\end{align*}
It follows that $f\ast g\in\mathrm{Lip}(\Sph^{d-1})$, and $\|f\ast g\|_{\mathrm{Lip}}\leq \|f\|_{\mathrm{Lip}}\|g\|_{\Lambda,1}.$

The finite dimensional space of real homogeneous harmonic polynomials of degree $n$ on $\R^d$ restricted to $\Sph^{d-1}$ is denoted by $\mathcal{H}^d_n$. These spaces are mutually orthogonal with respect to the inner product 
$$
\langle f,g\rangle_{\Sph^{d-1}} = \frac{1}{\omega_n}\int_{\Sph^{d-1}} f(x)g(x)d\sigma(x)
$$
and they densely span $L^2(\Sph^{d-1})$ (see e.g. \cite{dai2013approximation}, Theorem 2.2.2). Denoting by proj$_n$ the corresponding projection operators, we can associate the partial sum operators
$$
S_n f=\sum_{k=0}^n \mbox{proj}_nf.
$$
These are finite-rank and satisfy $S_nf = f\ast K_n$, where 
$$
K_n(t) = \sum_{k=0}^n \frac{k+\Lambda}{\Lambda} C_k^\Lambda(t)
$$
and $C_k^\Lambda$ are the Gegenbauer polynomials (\cite{dai2013approximation}, Proposition 2.2.1). Fix $\delta\geq d-1$ and consider the averages 
$$
K^\delta_n(t) = \frac{1}{A_n^\delta}\sum_{k=0}^n A_{n-k}^\delta \frac{k+\Lambda}{\Lambda} C_k^\Lambda(t),
$$
where $A_k^\delta = \binom{k+\delta}{k}=\frac{(\delta+k)(\delta+k-1)\dots(\delta+1)}{k!}$. These give rise to a sequence of finite-rank operators on $L^2(\Sph^{d-1})$ defined by $S_n^\delta: f\mapsto f\ast K_n^\delta$. Write $\Lambda_n^\delta:=\|S_n^\delta\|_1= \sup \{\|S_n^\delta h\|_1: h\in B_{L^1(\Sph^{d-1})}\}$.  By \cite[Theorem 2.4.3]{dai2013approximation}, for each $n\in \N$, $S_n^\delta$ is a nonnegative operator, which implies that $K_n^\delta(t)\geq 0$, $t\in [-1,1]$. It follows that 
$$
\Lambda_n^\delta=\sup_{\|f\|_1 \leq 1} \|f\ast K_n^\delta\|_1 \geq \|1\ast K_n^\delta\|_1 = \|K_n^\delta\|_{\Lambda,1}.
$$
It is clear that also $S_n^\delta(C(\Sph^{d-1}))\subset C(\Sph^{d-1})$ and that its restriction to $C(\Sph^{d-1})$ is continuous in the uniform norm. Moreover, by \cite[Corollary 2.4.5]{dai2013approximation}, when  $f$ is continuous (and in particular when it is Lipschitz) on $\Sph^{d-1}$, we have that $S_n^\delta f$ converges to $f$ uniformly. On the other hand, by Young's inequality, for each $f\in \mathrm{Lip}(\Sph^{d-1})$ we have that
$$
\|S_n^\delta f\|_{\mathrm{Lip}} = \|f\ast K_n^\delta \|_{\mathrm{Lip}} \leq  \|f\|_{\mathrm{Lip}} \| K_n^\delta \|_{\Lambda,1}\leq \Lambda_n^\delta \|f\|_{\mathrm{Lip}}. 
$$
To conclude, let $\varepsilon>0$. By \cite[Theorem2.4.4]{dai2013approximation}, there is some $n_\varepsilon$ such that $\Lambda_n^\delta<1+\varepsilon$, for each $n\geq n_\varepsilon$. Thus $S_n^\delta$, $n\geq n_\varepsilon$ satisfy conditions (a)--(d) in Proposition \ref{tool} and we are done. \end{proof}

\noindent{\bf Acknowledgement.} We would like to thank to Gilles Godefroy who suggested to us to study these problems.

\bibliographystyle{siam}
\bibliography{references}
\end{document}